\documentclass{amsart}
\usepackage{graphicx}
\usepackage{amsmath}
\usepackage{amscd}
\usepackage{mathtools}
\usepackage{enumerate}
\usepackage{tikz-cd}
\usepackage[all,cmtip]{xy}
\usepackage{tikz}
\usepackage{amssymb}
\usepackage{ascmac} 
\usepackage{color}

\newtheorem{dfn}{Definition}[section]
\newtheorem{thm}[dfn]{Theorem}

\newtheorem{lem}[dfn]{Lemma}

\theoremstyle{definition}
\newtheorem{exa}[dfn]{Example}
\newtheorem{que}[dfn]{Question}

\title[Linear automorphisms giving Galois points]{Linear automorphisms of smooth hypersurfaces giving Galois points}
\author{Taro Hayashi}
\address{
	(Taro Hayashi)
	Faculty of Agriculture,
	Kindai University,
	Nakamaticho 3327-204, Nara, Nara 631-8505, Japan
}
\email{haya4taro@gmail.com}
\date{\today}
\begin{document}
	
\maketitle
\begin{abstract}
Let $X$ be a smooth hypersurface $X$ of degree $d\geq4$ in a projective space $\mathbb P^{n+1}$.
We consider a projection of $X$ from $p\in\mathbb P^{n+1}$ to a plane $H\cong\mathbb P^n$. 
This projection induces an extension of function fields $\mathbb C(X)/\mathbb C(\mathbb P^n)$. 
The point $p$ is called a Galois point if the extension is Galois. 
In this paper, we will give a necessary and sufficient conditions for $X$ to have Galois points by using linear automorphisms.
\end{abstract}
Keywords: Smooth hypersurface; Automorphism; Galois point; Galois extension.
	
MSC2010: Primary 14J70; Secondary 12F10.
\section{Introduction}
In this paper, we work over ${\mathbb C}$. 
For an irreducible variety $Y$, let $\mathbb C(Y)$ be the function field of $Y$. 
Let $X$ be a smooth hypersurface of degree $d\geq 4$ in $\mathbb P^{n+1}$, $p$ be a point in $\mathbb P^{n+1}$, and $\pi_p:X\dashrightarrow H$ be a projection with center $p$ where $H$ is a hyperplane not containing $p$. 
We have an extension of function fields $\pi^{\ast}:\mathbb C(H)\rightarrow \mathbb C(X)$ such that  $[\mathbb C(X):\mathbb C(H)]=d-1$ (resp. $d$) if $p\in X$ (resp. $p\not\in X$).
The structure of this extension does not depend on the choice of $H$ but on the point $p$. 
We write $K_p$ instead of $\mathbb C(H)$.
Since $H\cong\mathbb P^n$, $K_p\cong \mathbb C(\mathbb P^n)$ as a field.

Let $Y$ be an irreducible variety $Y$.
Let $K$ be a non-trivial intermediate field between $\mathbb C(Y)$ and $\mathbb C$ such that $K$ is a purely transcendental extension of $\mathbb C$ with the transcendence degree $n$.
The field $K$ is called a maximal rational subfield if there is not a non-trivial intermediate field $L$ between $\mathbb C(Y)$ and $K$ such that $L$ is a purely transcendental extension of $\mathbb C$ with the transcendence degree $n$.

Let $X$ be a smooth hypersurface of degree $d\geq 4$ in $\mathbb P^{n+1}$.
If $n=1$, then the field $K_p$ is a maximal rational subfield of $\mathbb C(X)$ ([\ref{bio:1000}]).  
In the case where $n=2$ and $d=4$, if $p$ is not an outer Galois point of $X$,
then the field $K_p$ is a maximal rational subfield. 
If $d\geq 5$, then $K_p$ is always a maximal rational subfield. 
Please see [\ref{bio:1001},\ref{bio:200}] for details. 
\begin{dfn}$([\ref{bio:20},\ref{bio:21},\ref{bio:22}])$.
The point $p\in\mathbb P^{n+1}$ is called a Galois point for $X$ if the extension $\mathbb C(X)/K_p$ is Galois. 
Moreover, if $p\in X$ $($resp. $p\not\in X)$, then we call $p$ an inner $($resp. outer$)$ Galois point. 
\end{dfn}
Pay attention that if $n=1$ or $p\not\in X$, then $\pi_p$ is a morphism such that $\pi_p:X\rightarrow \mathbb P^n$ is a Galois cover of a variety.
\begin{thm}$([\ref{bio:20},\ref{bio:21},\ref{bio:22}])$.
Let $X$ is a smooth hypersurface of degree $d\geq 4$ in $\mathbb P^{n+1}$, and $p\in\mathbb P^{n+1}$ be a Galois point of $X$.	
Then the Galois group of $\mathbb C(X)/K_p$ is induced by a linear automorphism of $X$.
In addition, if $p$ is an inner (resp. outer) Galois point, then the Galois group of $\mathbb C(X)/K_p$ is a cyclic group of $d-1$ (resp. $d$) 
\end{thm}
\begin{dfn}
An automorphism $g$ of $X$ is called linear if there is an automorhism $h$ of $\mathbb P^{n+1}$ such that $h(X)=X$ and $h_{|X}=g$.
\end{dfn}
If $X$ is a smooth hypersurface of degree $d\geq 4$ in $\mathbb P^{n+1}$ and $(n,d)\not=(2,4)$,
then the automorphism group Aut$(X)$ of $X$ is a finite subgroup of the group PGL$(n+2,\mathbb{C})={\rm Aut}(\mathbb P^{n+1})$, for instance, see ([\ref{bio:102}]).
\begin{dfn}
Let $p\in\mathbb P^{n+1}$ is a Galois point of $X$.
An automorphism $g$ of $X$ is called an automorphism belonging to the Galois point $p$ if $g$ generates the Galois group of the Galois extension $\mathbb C(X)/K_p$.
%We write an automorphism belonging to the Galois point $p$ as $g_p$.
\end{dfn}
\begin{dfn}
Let $g$ be a linear automorphism of $X$.
A matrix $A$ is called a representation matrix of $g$ if $g=A$ in ${\rm PGL}(n+2,\mathbb{C})$.
\end{dfn}
A necessary and sufficient condition for a smooth hypersurface $X\subset\mathbb P^{n+1}$ to have Galois points is given by the defining equation of $X$ $([\ref{bio:20},\ref{bio:21},\ref{bio:22}])$.
For the case $n=1$, there is a sufficient condition for a smooth plane $X$ curve to have Galois points by the structure of the automorphism group ${\rm Aut}(X)$ as follows.
\begin{thm}\label{thm:1}$([\ref{bio:1}])$.	
Let $X$ be a smooth hypersurface of degree $d\geq 4$ in $\mathbb P^{n+1}$, and $g\in{\rm Aut}(X)$ be a linear automorphism of order $k(d-1)$ $(resp.\ kd)$ for $n,k\geq 1$. 
If $n=1$ and $k\geq2$, then $X$ has an inner $(resp.\ outer)$ Galois point $p$, and $g^k$ is an automorphism belonging to the Galois point $p$. 
%the Galois group of the Galois cover $\pi_p:X\rightarrow \mathbb P^1$.
\end{thm}
Smooth curves in $\mathbb P^2$ with Galois points are characterized by other methods as well [\ref{bio:11},\ref{bio:12},\ref{bio:13}].
There are smooth plane curves of degree $d$ with a linear automorphism of order $d-1$ or $d$ acting but without Galois points (see Examples \ref{exa:1} and \ref{exa:2}).
In addition, there is a smooth hypersurface $X$ of degree $d$ in $\mathbb P^4$ with a linear automorphism of order $(d-1)d$ acting but without Galois points (see Example \ref{exa:3}).
Therefore, Theorem \ref{thm:1} does not hold for all $n,k\geq1$. 

%Let ${\rm Aut}(X)$ be the automorphism group of $X$.
For $g\in$Aut$(X)$,
we set ${\rm Fix}(g):=\{x\in X\,|\,g(x)=x\ \}$, and 
we write the order of $g$ as ord$(g)$.
Recall that if $X$ is a smooth hypersurface and $(n,d)\not=(2,4)$,
then Aut$(X)$ is a subgroup of PGL$(n+2,\mathbb{C})$, i.e. all automorphisms of $X$ are linear.
In this paper, by using Fix$(g)$ and ord$(g)$,
we will study the case $k,n\geq1$ of Theorem \ref{thm:1}.
Our main results are Theorems \ref{thm:2}, \ref{thm:3}, \ref{thm:4}, and \ref{thm:5}.

Theorem \ref{thm:2}, is for $n=k=1$.
%which are a necessary and sufficient conditions for a smooth hypersurface $X\subset\mathbb P^{n+1}$ to has Galois points by Fix$(g)$ and ord$(g)$.
\begin{thm}\label{thm:2}
Let $X$ be a smooth plane curve degree $d\geq 4$, and $g$ be a linear automorphism of $X$.\\ 
$(1)$ If ${\rm ord}(g)=d-1$, 
then $\sharp|{\rm Fix}(g)|\not=2$ if and only if $X$ has an inner Galois point $p$, 
and $g$ is an automorphism belonging to the Galois point $p$.\\
%$G=\langle g_p\rangle$.\\
% is the Galois group of the Galois cover $\pi_p:C\rightarrow \mathbb P^1$.\\
$(2)$ If ${\rm ord}(g)=d$, 
then ${\rm Fix}(g)\not=\emptyset$ if and only if $X$ has an outer Galois point $p$, 
and $g$ is an automorphism belonging to the Galois point $p$.
\end{thm}
Theorem \ref{thm:3} is for $k=1$, $n\geq 2$, and an inner Galois point.
\begin{thm}\label{thm:3}
Let $X$ be a smooth hypersurface of degree $d\geq 4$ in $\mathbb P^{n+1}$, and $g\in{\rm Aut}(X)$ be a linear automorphism of order $d-1$.\\
(1) If $n=2$, then ${\rm Fix}(g)$ contains a curve $C'$ which is not a smooth rational curve if and only if 
$X$ has an inner Galois point $p$, and $g$ is an automorphism belonging to the Galois point $p$.\\
(2) If $n\geq 3$, then ${\rm Fix}(g)$ has codimension $1$ in $X$ if and only if
$X$ has an inner Galois point $p$, and $g$ is an automorphism belonging to the Galois point $p$.
\end{thm}
Theorem \ref{thm:4} is for $k=1$, $n\geq 2$, and an outer Galois point.
\begin{thm}\label{thm:4}
Let $X$ be a smooth hypersurface of degree $d$ in $\mathbb P^{n+1}$, and $g\in{\rm Aut}(X)$ be a linear automorphism of order $d$.
If $d\geq 2$, then ${\rm Fix}(g)$ has codimension $1$ in $X$ if and only if $X$ has an outer Galois point $p$, and $g$ is an automorphism belonging to the Galois point $p$.
\end{thm}
The following Theorem is for $n,k\geq 2$ and an inner Galois point.
\begin{thm}\label{thm:5}
Let $X$ be a smooth hypersurface of degree $d$ in $\mathbb P^{n+1}$, and $g\in{\rm Aut}(X)$ be a linear automorphism of order $k(d-1)$ for $k\geq 2$.\\
(1) If $n=2$ and $\sharp|{\rm Fix}(g)|\geq5$, then $X$ has an inner Galois point $p$, and $g^k$ is an automorphism belonging to the Galois point $p$.\\
(2) If $n\geq 3$ and ${\rm Fix}(g)$ has codimension $1$ or $2$ in $X$, then $X$ has an inner Galois point $p$, and $g^k$ is an automorphism belonging to the Galois point $p$.
\end{thm}
Theorem \ref{thm:5} does not fold for an outer Galois point (see Example \ref{exa:6}).
For $n=1$, the automorphism groups of curves with Galois points are classified ([\ref{bio:1},\ref{bio:8}]).
There are studies on automorphism groups of plane curves using Galois points ([\ref{bio:1},\ref{bio:10},\ref{bio:11},\ref{bio:12},\ref{bio:16},\ref{bio:17}]).
%Galois points of smooth hypersurfaces of the projective space $\mathbb P^{n+1}$ for $n\geq 2$ are examined ([\ref{bio:84},\ref{bio:85}]\ {\rm and\ see\ also}\ [\ref{bio:3}]).
For the case $n\geq 2$, determining whether $X$ has Galois points from the structure of ${\rm Aut}(X)$ may be an important issue.
\begin{que}\label{q:1}
For $n\geq 1$, is there a group $G_n$ satisfying the following condition ?
The condition: If the automorphism group ${\rm Aut}(X)$ of a smooth hypersurface $X$ of degree $d\geq 4$ in $\mathbb P^{n+1}$ has a subgroup $H$ which is isomorphic to $G$ as a group, then $X$ has a Galois point.
\end{que}
Theorem \ref{thm:1} is an answer to Question \ref{q:1} for the case $n=1$. 
However, our main theorems are not answers to Question \ref{q:1}, because they need the fixed points set.
Section 2 is preliminary. We will explain the basic facts of Galois point.
In section 3, we will show Theorems \ref{thm:2}, \ref{thm:3}, \ref{thm:4}, and \ref{thm:5}.
\section{Preliminary} 
Let $X$ be a smooth hypersurface of degree $d\geq4$ in $\mathbb P^{n+1}$. 
We denote the number of inner (resp. outer) Galois points of $X$ by $\delta(X)$ (resp. $\delta'(X)$). 
Here $[s]$ represents the integer part of $s\in\mathbb R$.
%Then the following is known.
\begin{thm}\label{thm:6}$([\ref{bio:20},\ref{bio:21},\ref{bio:22}])$. 
Let $X$ be a smooth hypersurface of degree $d\geq 4$ in $\mathbb P^{n+1}$. 
The following holds.\\	
(1) If $n=1$, then $\delta(X)=0,1$, or $4$, and  $\delta'(X)=0,1$, or $3$.
In particular, if $n=1$ and $d\geq 5$, then $\delta(X)=0$ or $1$.\\
(2) If $n\geq2$ and $d=4$, then $\delta(X)\leq 4([\frac{n}{2}]+1)$. 
In particular, if $n=2$ and $d=4$, then $\delta(X)=0,1,2,4$, or $8$.\\
(3) If $n\geq2$ and $d\geq5$, then $\delta(X)\leq [\frac{n}{2}]+1$.\\
(4) If $n\geq2$ and $d\geq4$, then $\delta'(X)\leq n+2$.
\end{thm}
The numbers of Galois points of normal hypersurfaces are investigated ([\ref{bio:5},\ref{bio:19}]).

%We describe the arrangement of Galois points (please see [\ref{bio:22}] for more detailed results).
%For two points $p,q\in\mathbb P^{n+1}$, let $L_{p,q}\subset\mathbb P^{n+1}$ be the line passing through two points $p$ and $q$.
%\begin{dfn}\label{dfn:1213}
%$([\ref{bio:22}])$.
%Let $X$ be a smooth hypersurface of degree $d\geq 4$ in $\mathbb P^{n+1}$. 
%A set of Galois points $\{p_0,\ldots,p_r\}$ of $X$ is said to be independent, if for any $0\leq i<j\leq r$,
%all the Galois points of $X$ lying on $L_{p_i,p_j}$ are just $p_i$ and $p_j$.
%\end{dfn}
%\begin{thm}\label{thm:1214}
%$([\ref{bio:22}])$.
%Let $X$ be a smooth hypersurface of degree $d\geq 4$ in $\mathbb P^{n+1}$, and $p_0,\ldots,p_r$ be independent Galois points of $X$.
%Then $p_0,\ldots,p_r$ are linearly independent, i.e. there is no projective subspace $V\subset \mathbb P^{n+1}$ such that dim\,$V<r$, and $p_i\in V$ for $0\leq i\leq r$.
%\end{thm}
%\begin{thm}\label{thm:1213}
%$([\ref{bio:22}])$.
%Let $X$ be a smooth hypersurface of degree $d\geq 5$ in $\mathbb P^{n+1}$. 
%Then the set of Galois points of $X$ is independent.
%\end{thm}
%\ref{bio:22}
The defining equations for smooth hypersurfaces with a Galois point are determined.
\begin{thm}\label{thm:7}$([\ref{bio:20},\ref{bio:21},\ref{bio:22}])$.
Let $X$ be a smooth hypersurface of degree $d\geq 4$ in $\mathbb P^{n+1}$. 
The following holds.\\
(1) $X$ has an inner Galois point $p$ if and only if by replacing the local coordinate system if necessary, $p=[1:0:\cdots:0]$ and  $X$ is defined by
\[ X_1X_0^{d-1}+F(X_1,\ldots,X_{n+1})=0.\]
(2) $X$ has an outer Galois point $p$ if and only if
by replacing the local coordinate system if necessary, $p=[1:0:\cdots:0]$ and $X$ is defined by
\[ X_0^d+F(X_1,\ldots,X_{n+1})=0.\]
\end{thm}
The definition equations with many Galois points are also studied (please see [\ref{bio:22}] for more detailed results). 

For a positive integer $l$, let  $I_l$ be the identity matrix of size $l$, and $e_l$ be a  primitive $l$-th root of unity.
Theorem \ref{thm:8} below is a rewrite of Theorem \ref{thm:7} from the viewpoint of a liner automorphism.
\begin{thm}\label{thm:8}
Let $X$ be a smooth hypersurface of degree $d$ in $\mathbb P^{n+1}$, $g\in{\rm Aut}(X)$ be a linear automorphism of order $d-1$ $(resp.\ d)$, and $A$ be a representation matrix of $g$.
There is a Galois point $p$ of $X$ such that $g$ is an automorphism belonging to the Galois point $p$ if and only if 
the matrix $A$ is conjugate to a matrix 
\[\begin{pmatrix}
a & 0\\
0&bI_{n+1}
\end{pmatrix}\]
such that $\frac{a}{b}=e_{d-1}$ $(resp.\ e_d)$.
In particular, if $A$ is conjugate to the above matrix, then 
the Galois point $p$ is the eigenvector corresponding to the eigenvalue $a$.
\end{thm}
From Theorem \ref{thm:8}, we see that the only if parts of Theorems \ref{thm:3} and \ref{thm:4} holds.

From here, we give examples of smooth hypersurfaces of degree $d$ without Galois points which have a linear automorphism such that the order is a multiple of $d-1$ or $d$.
As a corollary of Theorem \ref{thm:8}, we give the following two lemmas.
\begin{lem}\label{lem:1}
Let $X$ be a smooth hypersurface of degree $d\geq4$ in $\mathbb P^{n+1}$, $p\in\mathbb P^{n+1}$, and $g$ be an automorphism belonging to the Galois point $p$.
For any linear automorphism $h$ of $X$, $h(p)$ is also a Galois point of $X$, and $h\circ g\circ h^{-1}$ is an automorphism belonging to the Galois point $h(p)$.
In particular, if $p$ is an inner (resp. outer) Galois point, then $h(p)$ is also an inner (resp. outer) Galois point.
\end{lem}
\begin{proof}
By a linear automorphism $h\circ g\circ h^{-1}$ and Theorem \ref{thm:8}, $h(p)$ is a Galois point of $X$, and  
$h\circ g\circ h^{-1}$ is an automorphism belonging to the Galois point $h(p)$.
\end{proof}
\begin{lem}\label{lem:2}
Let $X$ be a smooth hypersurface of degree $d\geq4$ in $\mathbb P^{n+1}$, $p\in\mathbb P^{n+1}$, and $g$ be an automorphism belonging to the Galois point $p$.
For a linear automorphism $k$ of $X$ such that $k(p)=p$, we get that $k\circ g=g\circ k$.
\end{lem}
\begin{proof}
By Lemma \ref{lem:1}, $k\circ g\circ k^{-1}$ is an automorphism belonging to the Galois point $p$.
By Theorem \ref{thm:8}, $k\circ g\circ k^{-1}=g$.
\end{proof}
In Example \ref{exa:1}, we give an example of a smooth plane curve of degree $d$ with a linear automorphism of order $d-1$ but has no Galois points.
Before that, we prepare a lemma.
\begin{lem}\label{lem:3}
Let $A:=(a_{ij})$ be a diagonal $m\times m$ matrix such that $a_{ii}\not=a_{jj}$ for $1\leq i<j \leq m$. 	
For a $m\times m$ matrix $B:=(b_{ij})$, if $AB=BA$, then $B$ is a diagonal matrix. 
\end{lem}
\begin{proof}
We assume that $AB=BA$.
The $(i,j)$-th entry of the matrix $AB$ is $a_{ii}b_{ij}$.
The $(i,j)$-th entry of the matrix $BA$ is $a_{jj}b_{ij}$.
Since $a_{ii}\not=a_{jj}$ for $\leq i<j \leq m$, we get that $b_{ij}=0$ for $\leq i<j \leq m$.
Then the matrix $B$ is a diagonal matrix.
\end{proof}
\begin{exa}\label{exa:1}
Let $d$ be an even number of $6$  or more, and $X$ be a smooth curve in $\mathbb P^2$  defined by 
\[X_2^d+X^{d-1}_0X_2+X^{d-1}_1X_2+X_0^{\frac{d}{2}}X_1^{\frac{d}{2}}=0.\]
The curve $X$ has an automorphism $g$ of order $d-1$ such that the following matrix $A$ is a representation matrix of $g$:
\[A:=
\begin{pmatrix}
e_{d-1}^{\frac{d}{2}} &0 &0\\
0 & e_{d-1}^{\frac{d}{2}-1} &0\\
0 & 0 &1
\end{pmatrix}.\]
For $1\leq i<d-1$, we get that $\frac{d}{2}i\not \equiv0$ $({\rm mod}\ d-1)$, $(\frac{d}{2}-1)i\not\equiv0$ $({\rm mod}\ d-1)$, and $\frac{d}{2}i\not\equiv(\frac{d}{2}-1)i$ $({\rm mod}\ d-1)$.
We assume that $X$ has a Galois point $p\in  \mathbb P^2$.
By Lemma \ref{lem:1}, $g^j(p)$ is a Galois point for $1\leq j<d-1$.
By Theorem \ref{thm:6}, $\delta(X)\leq 4$ and $\delta'(X)\leq3$.
Since $d\geq 6$, $g^l(p)=p$ for some $1\leq l< d-1$.
Let $h\in{\rm Aut}(X)$ be an automorphism belonging to the Galois point $p$.
Since $g^l(p)=p$, the automorphism $g^l\circ h\circ g^{-l}$ is also an automorphism belonging to the Galois point $p$.
Then $g^l\circ h\circ g^{-l}=h^i$ for some $1\leq i<d-1$.
By Theorem \ref{thm:8}, we can take a representation matrix $B$ of $h$ such that 
\[CBC^{-1}=
\begin{pmatrix}
e_k &0 &0\\
0 & 1 &0\\
0 & 0 &1
\end{pmatrix}\]
 for some a matrix $C$ where if $p\in X$, then $k=d-1$, and if $p\not \in X$, then $k=d$.
By the equation $g^l\circ h\circ g^{-l}=h^i$, we get that $i=1$, and $A^lBA^{-l}=B$.
Since the diagonal entries of $A^l$ are different from each other, Lemma \ref{lem:3}, and $A^lBA^{-l}=B$, we get that $B$ is a diagonal matrix.
Since $h=B$ is an automorphism belonging to the Galois point $p$, and Theorem \ref{thm:8},
we get that 
\[p\in\{[1:0:0],[0:1:0],[0:0:1]\},\]
and the matrix $B$ is one of the following matrices
\[
\begin{pmatrix}
	a &0 &0\\
	0 & b &0\\
	0 & 0 &b
\end{pmatrix}, 
\begin{pmatrix}
	b & 0 &0\\
	0 & a &0\\
	0 & 0 &b
\end{pmatrix},\  {\rm and}\  
\begin{pmatrix}
	b &0 &0\\
	0 & b &0\\
	0 & 0 &a
\end{pmatrix}\]
where if $p\in X$, then $\frac{a}{b}=e_{d-1}$, and if $p\not \in X$, then $\frac{a}{b}=e_d$.
The defining equation of $X$ implies that $h=B$ is not an automorphism of $X$. 
This is a contradiction.
Therefore, $X$ does not have Galois points.
\end{exa}
Below is an example of a smooth plane curve of degree $d$ with a linear automorphism $d$ but has no Galois points.
\begin{exa}\label{exa:2}
Let $d_1$ and $d_2$ be integers greater than $4$ such that gcd$(d_1,d_2)=1$.
Let $d:=d_1d_2$, and $X$ be a smooth curve in $\mathbb P^2$  defined by 
\[X_0^d+X_1^d+X_2^d+X^{d_1}_0X^{d_2}_1X_2^{d-d_1-d_2}=0.\]
The curve $X$ has an automorphism $g$ of order $d$ such that the following matrix $A$ is a representation matrix of $g$:
\[A:=
\begin{pmatrix}
e_{d_1} &0 &0\\
0 & e_{d_2} &0\\
0 & 0 &1
\end{pmatrix}.\] 
For $1\leq i<d$, the diagonal entries of $A^i$ are different from each other.
As like Example \ref{exa:1},
we get that $X$ does not have Galois points.
\end{exa}
We give an example of a smooth surface $X$ of degree $d\geq 4$ in $\mathbb P^3$ such that $X$ has a linear automorphism $g$ of order $(d-1)d$ but has no Galois points.
\begin{exa}\label{exa:3}
Let $d_1\geq 5$ be an odd integer, and $d:=2d_1+1$.
Let $X$ be a smooth surface of degree $d$ in $\mathbb P^3$ defined by 
\[X_0^d+X_0^{\frac{d+1}{2}}X_1^{\frac{d-1}{2}}+X_0X_1^{d-1}+X_2^{d-1}X_3+X_2X_3^{d-1}=0.\]
The surface $X$ has an automorphism $g$ of order $(d-1)d$ such that the following matrix $A$ is a representation matrix of $g$
\[
A:=
\begin{pmatrix}
e_{\frac{(d-1)}{2}d}^{1-d}&0&0\\
0&e_{\frac{(d-1)}{2}d}&0\\
0&0&0&1\\
0&0&1&0
\end{pmatrix}.
\]
In addition,
the surface $X$ has an automorphism $h$ of order $(d-2)\frac{(d-1)}{2}d$ such that the following matrix $B$ is a representation matrix of $h$
\[
B:=
\begin{pmatrix}
	e_{\frac{(d-1)}{2}d}^{1-d}&0&0\\
	0&e_{\frac{(d-1)}{2}d}&0\\
	0&0&e_{d-2}&0\\
	0&0&0&e_{d-2}^{-1}
\end{pmatrix}.
\]
For $1\leq i<\frac{(d-1)}{2}d$, the diagonal entries of $B^i$ are different from each other.
By Theorem \ref{thm:6}, $\delta(X)\leq 2$ and $\delta'(X)\leq4$.
Since $\frac{(d-1)}{2}d\geq 5$, if $X$ has a Galois point, then there is a Galois point $p$ of $X$ such that $g^l(p)=p$ for some $1\leq l<\frac{(d-1)}{2}d$.
As like Example \ref{exa:1}, this is a contradiction.
Then $X$ does not have Galois points.
\end{exa}
From here, based on [\ref{bio:1}], we explain the orders of automorphisms of smooth plane curves of degree $d\geq 4$.
Let $X$ be a smooth plane curve of degree $d\geq4$, and $g$ be an automorphism of $X$.
By replacing the local coordinate system if necessary, we may assume that $g$ is defined by a diagonal matrix, i.e.  $g=
\begin{pmatrix}
\alpha &0 &0\\
0 &\beta &0\\
0 & 0 &\gamma
\end{pmatrix}$. 
Let 
\[ n(g):=\sharp|{\rm Fix}(g)\cap\{[1:0:0],[0:1:0],[0:0:1]\}|.\]
Since $g$ is defined by a diagonal matrix, $n(g)=X\cap\{[1:0:0],[0:1:0],[0:0:1]\}$.
Then $n(g)=0,1,2$, or $3$.
The following Theorem \ref{thm:20} determines orders of cyclic groups acting on smooth plane curves. 
Theorem \ref{thm:2} is shown by Theorems \ref{thm:8} and \ref{thm:20}.

For a smooth hypersurface $X\subset\mathbb P^{n+1}$,
orders of automorphisms of $X$ and the structure of the group Aut$(X)$ are studied for $n\geq 1$ ([\ref{bio:2},\ref{bio:8},\ref{bio:100},\ref{bio:101},\ref{bio:103},\ref{bio:104}]).
Also, as in [\ref{bio:191},\ref{bio:103}], the structures of subgroups of Aut$(X)$ are also investigated based on the way they act on $X$.
In this paper, we examine automorphisms of $X$ that give Galois points.
At the end of this section, we classify abelian groups acting on smooth plane curves (Theorem \ref{thm:21}).
\begin{thm}\label{thm:20}$([\ref{bio:1}])$.
Let $X$ be a smooth curve of degree $d\geq 4$ in $\mathbb P^2$, and $g$ be an automorphism of $X$.
By replacing the local coordinate system if necessary, the order of $g$ and a representation matrix of $g$ are one of Table 1.
\end{thm}
\begin{table}[h]
\caption{Cyclic groups of smooth plane curves of degree $d\geq 4$}
\label{table:data_type}
\centering
\begin{tabular}{|c|c|c|c|}
\hline
No.&$n(g)$&Order $l$ of $g$&Representation matrix of $g$ \\ \hline 
1&$0$&$l$ divides $d$& 
$\begin{pmatrix}
e_l^s &0 &0\\
0 &e_l^t &0\\
0 & 0 &1
\end{pmatrix}$\\ \hline
2&$1$&$l$ divides $d-1$& 
$\begin{pmatrix}
e_l &0 &0\\
0 &1 &0\\
0 & 0 &1
\end{pmatrix}$\\ \hline
3&$1$&$l$ divides $(d-1)d$& 
$\begin{pmatrix}
e_l &0 &0\\
0 &e_l^{1-d} &0\\
0 & 0 &1
\end{pmatrix}$\\ \hline
4&$2$&$l$ divides $d-1$& 
$\begin{pmatrix}
e_l^s &0 &0\\
0 & e_l^t &0\\
0 & 0 &1
\end{pmatrix}$\\ \hline
5&$2$&$l$ divides $(d-1)^2$& 
$\begin{pmatrix}
e_l^{1-d} &0 &0\\
0 & e_l &0\\
0 & 0 &1
\end{pmatrix}$\\ \hline
6&$2$&$l$ divides $(d-2)d$& 
$\begin{pmatrix}
e_l &0 &0\\
0 & e_l^{1-d} &0\\
0 & 0 &1
\end{pmatrix}$\\ \hline
7&$3$&$l$ divides $d-1$& 
$\begin{pmatrix}
e_l &0 &0\\
0 &1 &0\\
0 & 0 &1
\end{pmatrix}$\\ \hline
8&$3$&$l$ divides $d^2-3d+3$& 
$\begin{pmatrix}
e_l &0 &0\\
0 & e_l^{d-1} &0\\
0 & 0 &1
\end{pmatrix}$\\ \hline
\end{tabular}
\end{table}
\begin{thm}\label{thm:21}
Let $X$ be a smooth plane curve of degree $d\geq4$, and $G$ be an abelian subgroup of Aut$(X)$.
If $G$ is not a cyclic group, then $G$ is isomorphic to a subgroup of $\mathbb Z/d\mathbb Z^{\oplus 2}$ as a group.
\end{thm}
\begin{proof}
Since $d\geq4$, $G$ is a finite subgroup of ${\rm PGL}(3,\mathbb C)$.
Let $l:=$max$\{{\rm ord}(k)\,|\,k\in G\}$.
We take an element $g\in G$ such that ord$(g)=l$.  
By replacing the local coordinate system if necessary, we may assume that $g$ is defined by a diagonal matrix.

First, we assume that
$g=\begin{pmatrix}
	\alpha &0 &0\\
	0 &\alpha &0\\
	0 & 0 &\beta
\end{pmatrix}$ where $\alpha,\beta\in\mathbb C^{\ast}$.
For simplicity, we may assume that 
$\alpha=e_l$ and $\beta=1$.
Let $h$ be an element of $G$ such that $h\not\in \langle g\rangle$, and $A:=(a_{ij})_{1\leq i,j\leq3}$ be a representation matrix of $h$.
Since $g\circ h=h\circ g$, we get that 
\[ 
\begin{pmatrix}
u &0 &0\\
0 &u &0\\
0 & 0 &u
\end{pmatrix}
\begin{pmatrix}
e_l &0 &0\\
0 & e_l &0\\
0 & 0 &1
\end{pmatrix}
\begin{pmatrix}
a_{11} &a_{12} &a_{13}\\
a_{21} &a_{22} &a_{23}\\
a_{31} &a_{32} &a_{33}
\end{pmatrix}
=
\begin{pmatrix}
a_{11} &a_{12} &a_{13}\\
a_{21} &a_{22} &a_{23}\\
a_{31} &a_{32} &a_{33}
\end{pmatrix}
\begin{pmatrix}
e_l &0 &0\\
0 & e_l &0\\
0 & 0 &1
\end{pmatrix},
\]
and hence 
$
\begin{pmatrix}
ue_la_{11} &ue_la_{12} &ue_la_{13}\\
ue_la_{21} &ue_la_{22} &ue_la_{23}\\
ua_{31} &ua_{32} &ua_{33}
\end{pmatrix}
=
\begin{pmatrix}
e_la_{11} &e_la_{12} &a_{13}\\
e_la_{21} &e_la_{22} &a_{23}\\
e_la_{31} &e_la_{32} &a_{33}
\end{pmatrix}$
for $u\in \mathbb C^{\ast}$.
If $u\not=1$, then $a_{11}=a_{12}=a_{21}=a_{22}=a_{33}=0$.
Since $A$ is invertible, this is a contradiction.
Therefore, $u=1$.
Then $a_{13}=a_{23}=a_{31}=a_{32}=0$.
This means that
there is an injective homomorphism 
\[\vartheta:G\ni k \mapsto C\in {\rm GL}(2,\mathbb C)\ \  {\rm such\ that}\ \  k=\begin{pmatrix}
C &0\\
0 &1
\end{pmatrix}.\]
% is a representation matrix of $k$.
Since $\vartheta(G)$ is an abelian but not cyclic subgroup of ${\rm GL}(2,\mathbb C)$, 
there are two matrices $S_1,S_2\in {\rm GL}(2,\mathbb C)$ such that $\vartheta(G)=\langle S_1\rangle \oplus\langle S_2\rangle $.
In order to show $G\subset \mathbb Z/d\mathbb Z^{\oplus 2}$, we only show that ord$(g')$ is a divisor of $d$ for any $g'\in G$.
Since $G\cong \vartheta(G)=\langle S_1\rangle \oplus\langle S_2\rangle$, by replacing the local coordinate system if necessary, we may assume that $G$ is generated by two diagonal matrices.
We assume that $p:=[1:0:0]\in X$.
Since $G$ is generated by diagonal matrices, we get that $p\in {\rm Fix}(g)$ for any $g\in G$.
Since dim\,$X=1$, and $X$ is smooth, we get that $G$ is a cyclic group.
This contradicts that $G$ is not a cyclic group.
Therefore, we get that $[1:0:0]\not\in X$.
Similarly, we get that $[0:1:0],[0:0:1]\not\in X$.
Since $[1:0:0],[0:1:0],[0:0:1]\not\in X$, $X$ is defined by 
\[ aX^d+bY^d+cZ^d+\sum_{i=0}^{d-1}F_{d-i}(Y,Z)X^i=0\]
where $abc\not=0$, $F_{d-i}(Y,Z)$ is a homogeneous polynomial of degree $d-i$ for $0\leq i\leq d-1$, and $F_0(Y,Z)$ has no $Y^d$ and $Z^d$ terms.
Then since $G$ is generated by diagonal matrices, we get that ord$(g')$ is a divisor of $d$ for any $g'\in G$.
Therefore, $G$ is a subgroup of $\mathbb Z/d\mathbb Z^{\oplus 2}$.
\\

Next, we assume that there is not an element $g'\in G$ such that a representation matrix of $g'$ is conjugate to 
$
\begin{pmatrix}
\alpha' &0 &0\\
0 &\alpha' &0\\
0 & 0 &\beta'
\end{pmatrix}
$ where $\alpha',\beta'\in\mathbb C^{\ast}$.
Then we may assume that
$g=\begin{pmatrix}
e_l^s &0 &0\\
0 & e_l^t &0\\
0 & 0 &1
\end{pmatrix}$
where $e_l^s\not=e_l^t$, $e_l^s\not=1$, and $e_l^t\not=1$.
Let $h$ be an element of $G$ such that $h\not\in \langle g\rangle$, and $A:=(a_{ij})_{1\leq i,j\leq3}$ be a representation matrix of $h$.
Since $g\circ h=h\circ g$,
$
\begin{pmatrix}
ue_l^sa_{11} &ue_l^sa_{12} &ue_l^sa_{13}\\
ue_l^ta_{21} &ue_l^ta_{22} &ue_l^ta_{23}\\
ua_{31} &ua_{32} &ua_{33}
\end{pmatrix}
=
\begin{pmatrix}
e_l^sa_{11} &e_l^ta_{12} &a_{13}\\
e_l^sa_{21} &e_l^ta_{22} &a_{23}\\
e_l^sa_{31} &e_l^ta_{32} &a_{33}
\end{pmatrix}$
for $u\in \mathbb C^{\ast}$.
If $a_{ii}\not=0$ for some $1\leq i\leq 3$, then $u=1$.
Since $e_l^s\not=e_l^t$, $e_l^s\not=1$, and $e_l^t\not=1$, we get that $a_{ij}=0$ for $i\not=j$, i.e. $A$ is a diagonal matrix.
Since ord$(h)$ divides $l$, and $g$ and $h$ are defined by diagonal matrices, 
we get that $\langle g,h\rangle$ contains an automoprhism $k$ such that a representation matrix of $k$ is conjugate to 
$
\begin{pmatrix}
\alpha &0 &0\\
0 &\alpha &0\\
0 & 0 &\beta
\end{pmatrix}$ where $\alpha,\beta\in\mathbb C^{\ast}$.
This contradicts the assumption for $G$.
Therefore, $a_{ii}=0$ for any $i=1,2,3$.
Since $A$ is invertible, $a_{12}\not=0$ or $a_{13}\not=0$.
We assume that $a_{12}a_{13}\not=0$.
Then $ue_l^s=e_l^t$ and $ue_l^s=1$, and hence we get that $e_l^t=1$.
This contradicts the assumption that $e_l^t\not=1$.
Therefore, $a_{12}a_{13}=0$ and $(a_{12},a_{13})\not=(0,0)$.
In the same way, 
$a_{21}a_{23}=a_{31}a_{32}=0$, $(a_{21},a_{23})\not=(0,0)$, and $(a_{31},a_{32})\not=(0,0)$.
Since $A$ is invertible,
\[
A=
\begin{pmatrix}
0 &a_{12} &0\\
0 &0 &a_{23}\\
a_{31} &0 &0
\end{pmatrix}\  {\rm or}\   
\begin{pmatrix}
0 &0 &a_{13}\\
a_{21} &0 &0\\
0 &a_{32} &0
\end{pmatrix}.
\]
If $A$ is the former,
then $ue_l^s=e_l^t$, $ue_l^t=1$, and $u=e_l^s$.
Therefore, we get that ${e_l^s}^3={e_l^t}^3=u^3=1$.
In the same way, for the latter case, we get that ${e_l^s}^3={e_l^t}^3=u^3=1$.
Therefore, we may assume that
$ 
g=\begin{pmatrix}
e_3^2 &0 &0\\
0 & e_3 &0\\
0 & 0 &1
\end{pmatrix}$, and for an automorphism $k\in G\backslash\langle g\rangle$, 
$k$ is defined by a matrix of the form:
\[
\begin{pmatrix}
	0 &b_{12} &0\\
	0 &0 &b_{23}\\
	b_{31} &0 &0
\end{pmatrix}\  {\rm or}\   
\begin{pmatrix}
	0 &0 &b_{13}\\
	b_{21} &0 &0\\
	0 &b_{32} &0
\end{pmatrix}.
\]
Note that the square of the former (resp. latter) form of the matrix is of the latter (resp. former) form of the matrix.
From here, we show that $G\cong\mathbb Z/3\mathbb Z^{\oplus 3}$, and the degree $d$ of $X$ is a multiple of $3$.
We assume that there are two automorphisms $h_1,h_2\in G$ such that 
\[ h_1=
	\begin{pmatrix}
	0 &a &0\\
	0 &0 &b\\
	c & 0 &0
	\end{pmatrix},\ \
h_2=
	\begin{pmatrix}
	0 &a' &0\\
	0 &0 &b'\\
	c' &0 &0
	\end{pmatrix},\]
and $h_1\not\in\langle h_2\rangle$.
%where $a,b,c,a',b',c'\in \mathbb C^{\ast}$.
Then 
\[
h_1^2\circ h_2=\begin{pmatrix}
	0 &a &0\\
	0 &0 &b\\
	c & 0 &0
	\end{pmatrix}
	\begin{pmatrix}
	0 &a &0\\
	0 &0 &b\\
	c & 0 &0
	\end{pmatrix}
	\begin{pmatrix}
	0 &a' &0\\
	0 &0 &b'\\
	c' &0 &0
	\end{pmatrix}=
	\begin{pmatrix}
	abc' &0 &0\\
	0 &a'bc &0\\
	0 & 0 &abc'
	\end{pmatrix}.
\]
Since $G$ is abelian, and ord$(h_i)=3$ for $i=1,2$, we get that ord$(h_1^2\circ h_2)=3$.
Since ord$(g)=3$, and the assumption for $G$, we get that $h_1^2\circ h_2\in \langle g\rangle$. 
Therefore, $G=\langle g, h\rangle\cong \mathbb Z/3\mathbb Z^{\oplus 3}$ where 
\[
g=\begin{pmatrix}
e_3^2 &0 &0\\
0 & e_3 &0\\
0 & 0 &1
\end{pmatrix}\   {\rm and}\   
h=
\begin{pmatrix}
0 &a &0\\
0 &0 &b\\
c & 0 &0
\end{pmatrix}.
\]
Since $h([1:0:0])=[0:0:1]$ and $h^2([1:0:0])=[0:1:0]$,
if $\{[1:0:0],[0:1:0],[0:0:1]\}\cap X\not=\emptyset$, 
then $\{[1:0:0],[0:1:0],[0:0:1]\}\subset X$, i.e. $n(g)=3$.
By Table 1 and a representation matrix of $g$, we get that $3$ divides $d$.
Then $G$ is a subgroup of $\mathbb Z/d\mathbb Z^{\oplus 2}$.
We assume that $\{[1:0:0],[0:1:0],[0:0:1]\}\cap X=\emptyset$. 
By Table 1 and a representation matrix of $g$, 
we get that ord$(g)=3$ divides $d^2-3d+3$, and hence $3$ divides $d$.
Therefore, $G$ is a subgroup of $\mathbb Z/d\mathbb Z^{\oplus 2}$.
\end{proof}
\section{Proof of main theorems}
First, we will show Theorem \ref{thm:2} (Theorem \ref{thm:9}).
Theorem \ref{thm:2} is immediately followed by Theorems \ref{thm:8} and \ref{thm:20}.
\begin{thm}\label{thm:9}
Let $X$ be a smooth plane curve degree $d\geq 4$, and $g$ be an automorphism of $X$.\\
$(1)$ If ${\rm ord}(g)=d-1$ and $\sharp|{\rm Fix}(g)|\not=2$, then $X$ has an inner Galois point $p$, and $g$ is an automorphism belonging to the Galois point $p$.\\
$(2)$ If ${\rm ord}(g)=d$ and ${\rm Fix}(g)\not=\emptyset$, then $X$ has an outer Galois point $p$, and $g$ is an automorphism belonging to the Galois point $p$.
\end{thm}
\begin{proof}
Since $d\geq4$, ${\rm Aut}(X)$ is a subgroup of ${\rm PGL}(3,\mathbb C)$.
We will show (1) of this theorem.
Since ${\rm ord}(g)=d-1$, by replacing the local coordinate system if necessary, we may assume that $g$ is defined by a diagonal matrix $A$ such that $A$ is one of no.2, 3, 4, 5, and 7 of Table 1.
By Theorem \ref{thm:8}, if $A$ is one of no.2, 3, 5, and 7 of Table 1, then $X$ has an inner Galois point $p$, and $g$ is an automorphism belonging to the Galois point $p$.
We assume that $A$ is no.4 of Table 1, i.e. $A=
\begin{pmatrix}
e_{d-1}^s &0 &0\\
0 & e_{d-1}^t &0\\
0 & 0 &1
\end{pmatrix}$ where $1\leq s,t<d-1$.
Then $X\cap \{[1:0:0],[0:1:0],[0:0:1]\}\subset{\rm Fix}(g)$ and $\sharp|X\cap \{[1:0:0],[0:1:0],[0:0:1]\}|=2$.
Since $\sharp|{\rm Fix}(g)|\not=2$, $\sharp|{\rm Fix}(g)|\geq3$.
Then we get that $s=t$, $s=1$, or $t=1$.
By Theorem \ref{thm:8}, $X$ has an inner Galois point $p$, and $g$ is an automorphism belonging to the Galois point $p$.
In the same way, we get (2) of this theorem.
\end{proof}
Let $X$ be a smooth hypersurface of degree $d\geq 4$ in $\mathbb P^{n+1}$, $p$ be a point in $\mathbb P^{n+1}$.
Recall that $\pi_p:X\dashrightarrow H$ is a projection with center $p$ where $H$ is a hyperplane not containing $p$. 

The following result is obtained for an inner Galois point ([\ref{bio:20}]).
\begin{thm}\label{cro:0}([\ref{bio:20}]).
Let $X$ be a smooth plane curve degree $d\geq 4$, and $\mathbb C(X)$ be the function field of $X$, and $k\subset \mathbb C(X)$ be a subfield such that $k$ is isomorphic to $\mathbb C(\mathbb P^1)$ as a field.
If $\mathbb C(X)/k$ is a Galois extension of degree $d-1$, then $X$ has an inner Galois point $p$, and the Galois extension $\mathbb C(X)/k$ is induced by $\pi_p:X\rightarrow \mathbb P^1$, i.e. $k=\pi_p^{\ast}(\mathbb C(\mathbb P^1))$.
\end{thm}
In the case of the outer Galois point, by Example \ref{ex1}, we see that the same result as in Theorem \ref{cro:0} does not hold.
\begin{exa}\label{ex1}
Let $X$ be a smooth curve of degree $4$ in $\mathbb P^2$ defined by 
\[X_0^4+X_1^4+X_2^4=0\]
which is called the Fermat curve of degree $4$.
The $X$ has two automorphism $g_1$ and $g_2$ of order $2$ such that 
the followimg matrices $A_1$ and $A_2$ are representation matrices of $g_1$ and $g_2$, respectively
\[
A_1:=
\begin{pmatrix}
-1&0&0\\
0&1&0\\
0&0&1
\end{pmatrix}\ {\rm and}\ 
A_2:=
\begin{pmatrix}
1&0&0\\
0&-1&0\\
0&0&1
\end{pmatrix}.
\]
Let $G$ be the subgroup of Aut$(X)$ generated by $g_1$ and $g_2$, and $g_3:=g_1\circ g_2\in G$.
Then $G\cong \mathbb Z/2\mathbb Z^{\oplus 2}$, and $\sharp|{\rm Fix}(g_i)|=4$ for $i=1,2,3$.

Let $G_x:=\{g\in G:\,g(x)=x\}$.
% be the stabilizer subgroup of $G$ with respect to $x\in X$.
For a smooth curve $C$, we write the genus of $C$ as $g(C)$.
By the Riemann-Hurwitz formula, 
\[ 2-2g(X)+\sum_{x\in X}(\sharp|G_x|-1)=\sharp|G|(2-2g(X/G)).\]
Since $X$ is a smooth plane curve of degree $4$, we get that $2-2g(X)=4(3-4)=-4$.
Then 
\[ 2-2g(X)+\sum_{x\in X}(\sharp|G_x|-1)=-4+12=8.\]
Since $\sharp|G|=4$, and the Riemann-Hurwitz formula, we get that $g(X/G)=0$, and hence $X/G\cong \mathbb P^1$.
Let $p:X\rightarrow X/G$ be the quotient morphism.
Since $G$ is not cyclic group, the Galois extension $\mathbb C(X)/p^{\ast}\mathbb C(\mathbb P^1)$ is not induced by a Galois point of $X$.
\end{exa}
The following theorem shows that similar results hold for an outer Galois point under the assumption of a cyclic extension.
\begin{thm}\label{cro:1}
Let $X$ be a smooth plane curve degree $d\geq 4$, and $\mathbb C(X)$ be the function field of $X$, and $k\subset \mathbb C(X)$ be a subfield such that $k$ is isomorphic to $\mathbb C(\mathbb P^1)$ as a field.
If $\mathbb C(X)/k$ is a cyclic extension of degree $d$, then $X$ has an outer Galois point $p$, and the cyclic extension $\mathbb C(X)/k$ is induced by $\pi_p:X\rightarrow \mathbb P^1$, i.e. $k=\pi_p^{\ast}(\mathbb C(\mathbb P^1))$.
\end{thm}
\begin{proof}
Since $X$ is a smooth curve, there is a cyclic subgroup $G$ of ${\rm Aut}(X)$ such that $X/G\cong \mathbb P^1$, and $k=p^{\ast}\mathbb C(\mathbb P^1)$ where $p:X\rightarrow X/G$ be the quotient morphism.
Since $d\geq4$, $G$ is a subgroup of ${\rm PGL}(3,\mathbb C)$.
Let $g$ be a generator of $G$.
By replacing the local coordinate system if necessary, we assume that there is a  diagonal matrix $A$ such that $A$ is a representation matrix of $g$.
Since ord$(g)=d$ and Theorem \ref{thm:9}, we only show that Fix$(g)\not=\emptyset$.

We assume that Fix$(g)=\emptyset$.
By Theorem \ref{thm:8}, that is, by the no.1 of Table 1, we may assume  that
$A=
\begin{pmatrix}
	e_d^s &0 &0\\
	0 & e_d^t &0\\
	0 & 0 &1
\end{pmatrix}$.
Since Fix$(g)=\emptyset$, 
$X\cap \{[1:0:0],[0:1:0],[0:0:1]\}=\emptyset$.
Then if Fix$(g^i)\not=\emptyset$ for some $1<i<d$, then $\sharp|{\rm Fix}(g^i)|=d$.
By the Riemann-Hurwitz formula and $C/G\cong \mathbb P^1$,  
\[ 2-2g(X)+\sum_{x\in X}(\sharp|G_x|-1)=2\sharp|G|=2d\]
Since $X$ is a smooth plane curve of degree $d$, we get that $2-2g(X)=d(3-d)$, and hence
By the matrix $A$, we get that ${\rm Fix}(g^i)\backslash \{[1:0:0],[0:1:0]\}\not=\emptyset$ if and only if $(e_{d-1}^{si}-e_{d-1}^{ti})(e_{d-1}^{si}-1)(e_{d-1}^{ti}-1)=0$ for $1< i<d$.
We define subgroups $G_1$, $G_2$, and $G_3$ of $G$ as follows:
\[ G_1:=\{g\in G\,|\,{\rm a\ representation\ matrix\ of}\ g\ {\rm is}\ \begin{pmatrix}
\alpha &0 &0\\
0 &1 &0\\
0 &0 &1
\end{pmatrix}\ {\rm for\ some}\ \alpha\in\mathbb C^{\ast}
\}.\]
\[ G_2:=\{g\in G\,|\,{\rm a\ representation\ matrix\ of}\ g\ {\rm is}\ \begin{pmatrix}
1 &0 &0\\
0 &\beta &0\\
0 &0 &1
\end{pmatrix}\ {\rm for\ some}\ \beta\in\mathbb C^{\ast}
\}.\]
\[G_3:=\{g\in G\,|\,{\rm a\ representation\ matrix\ of}\ g\ {\rm is}\ \begin{pmatrix}
\gamma &0 &0\\
0 &\gamma &0\\
0 &0 &1
\end{pmatrix}\ {\rm for\ some }\ \gamma\in\mathbb C^{\ast}
\}.\]
We set $a:=\sharp|G_1|$, $b:=\sharp|G_2|$, and $c:=\sharp|G_3|$.
Then $G_i\cap G_j=\{{\rm id}_X\}$ for $1\leq i<j\leq 3$, and ${\rm Fix}(g^i)\not=\emptyset$ if and only if $g^i\in \bigcup_{j=1}^3G_j$ for $1<i<d$.
Then
\[ (d-1)d=\sum_{x\in C}(\sharp|G_x|-1)=d(a+b+c-3).\] 
Therefore, 
\[ d+2=a+b+c.\] 
For simplicity, we may assume that $a\leq b\leq c$.
Since $d+2=a+b+c$, $1<c$.
Since $G_2\cap G_3=\{{\rm id}_X\}$ and $\sharp|G|=d$, we get that $bc|d$.
By the equation $d+2=a+b+c$, we get that $bc+2\leq a+b+c\leq b+2c$, and hence $(b-2)(c-1)\leq 0$.
Since $1<c$, $b\leq 2$.
If $b=2$, then by the equation $bc+2\leq a+b+c$, we get that $a=b=c$.
Since $G_i\cap G_i=\{{\rm id}_X\}$ for $1\leq i<j\leq3$, we get that $\mathbb Z_2^{\oplus 2}\cong \langle G_i,G_j\rangle\subset G$ where $1\leq i<j\leq3$, and
 $\langle G_i,G_j\rangle$ is the subgroup of $G$ generated by $G_i$ and $G_j$.
This contradicts that $G$ is a cyclic group.
If $b=1$, then $a=1$ and $c=d$.
This implies that $G=\langle g\rangle =G_3$. 
This contradicts that $G=\langle g\rangle$ and Fix$(g)=\emptyset$.
Therefore,  Fix$(g)\not=\emptyset$. 
By Theorem \ref{thm:9},  $X$ has an outer Galois point $p$, and $g$ is an automorphism belonging to the Galois point $p$.
\end{proof}
From here, we will study $X\subset \mathbb P^{n+1}$ for $n\geq 2$.
First, we give Examples \ref{exa:4} and \ref{exa:5} which imply that Corollary \ref{cro:1} does not hold for $n=2$.
\begin{exa}\label{exa:4}
Let $X$ be a smooth surface of degree $4$ in $\mathbb P^3$ defined by 
\[X_0^3X_2+X_1^3X_3+X_2^4+X_3^4=0.\]
The surface $X$ has an automorphism $g$ of order $3$ such that
\[g=
\begin{pmatrix}
e_3&0&0&0\\
0&e_3^2&0&0\\
0&0&1&0\\
0&0&0&1
\end{pmatrix}.\]
Then Fix$(g)$ contains a smooth rational curve.
Since the degree of $X$ is $4$, $X$ is a $K3$ surface.
Since Fix$(g)$ contains a curve, $g$ is a non-symplectic automorphism of order $3$.
Then the quotient space $Y:=X/\langle g \rangle$ is rational.
Let $q:X\rightarrow Y$ be the quotient morphism.
Since $Y$ is rational $k:=q^{\ast}\mathbb C(Y)\cong \mathbb C(\mathbb P^2)$ as a field.
However, by Theorem \ref{thm:8}, there is no a Galois point $p$ of $X$ such that $g$ is an automorphism belonging to the Galois point $p$.
In other words, there is no a Galois point $p$ of $X$ such that $k=\pi_p^{\ast}\mathbb C(\mathbb P^2)$.
Pay attention that $X$ has Galois points, and $\delta(X)=8$ ([\ref{bio:21}]).
\end{exa}
\begin{exa}\label{exa:5}
Let $X$ be a smooth surface in $\mathbb P^3$ defined by 
\[X_0^6+X_1^6+X_2^6+X_3^6+X_0^2X_1^3X_2+X_2^3X_3^3=0.\]
The surface $X$ has an automorphism $g$ of order $6$ such that
\[g=
\begin{pmatrix}
-1&0&0&0\\
0&e_3&0&0\\
0&0&1&0\\
0&0&0&1
\end{pmatrix}.\]
Fix$(g^3)=\{X_0=0\}\cap X:=H_1$ and Fix$(g^2)=\{X_1=0\}\cap X:=H_2$ are smooth curves, and Fix$(g)=H_1\cap H_2$.
Then the quotient space $Y:=X/\langle g\rangle$ is smooth.
Let $p:X\rightarrow Y$ be the quotient morphism, and $\mathcal O_X(1):=\mathcal O_{\mathbb P^{3}}(1)$ be the ample line bundle.
By the ramification formula, $K_X=p^{\ast}K_Y+H_1+2H_2$, and hence $p^{\ast}K_Y=K_X-H_1-2H_2$.
Since $K_X=\mathcal O_X(2)$, and $\mathcal O_X(H_i)=\mathcal O_X(1)$ for $i=1,2$,
we get that $p^{\ast}\mathcal O_Y(-K_Y)=\mathcal O_X(1)$ is ample.
Since the morphism $p:X\rightarrow Y$ is finite, $-K_Y$ is ample.
Since $Y$ is a smooth surface, $Y$ is rational, and hence $k:=q^{\ast}\mathbb C(Y)\cong \mathbb C(\mathbb P^2)$ as a field.
However, by Theorem \ref{thm:8}, there is no a Galois point $p$ of $X$ such that $k=\pi_p^{\ast}\mathbb C(\mathbb P^2)$.
\end{exa}
We will show  Theorems \ref{thm:3} and \ref{thm:4} (Theorem \ref{thm:10}).
Recall that for a smooth hypersurface $X\subset\mathbb P^{n+1}$ of degree $d\geq 4$, if $(n,d)\not=(2,4)$, then all automorphisms of $X$ are linear.
\begin{thm}\label{thm:10}
Let $X$ be a smooth hypersurface of degree $d\geq 4$ in $\mathbb P^{n+1}$, and $g$ be a linear automorphism of $X$.\\
(1) If $n=2$, ${\rm ord}(g)=d-1$, and ${\rm Fix}(g)$ contains a curve $C'$ which is not a smooth rational curve, 
then $X$ has an inner Galois point $p$, and $g$ is an automorphism belonging to the Galois point $p$.\\
(2) If $n\geq 3$, ${\rm ord}(g)=d-1$, and ${\rm Fix}(g)$ has codimension $1$ in $X$, then $X$ has an inner Galois point $p$, and $g$ is an automorphism belonging to the Galois point $p$.\\
(3) If $n\geq 2$, ${\rm ord}(g)=d$, and ${\rm Fix}(g)$ has codimension $1$ in $X$, then $X$ has an outer Galois point $p$, and $g$ is an automorphism belonging to the Galois point $p$.
\end{thm}
\begin{proof}
By replacing the local coordinate system if necessary, we may assume that
\[g=
\begin{pmatrix}
a_{i_1}I_{i_1} & &\\
&\ddots &\\
 &  &a_{i_m}I_{i_m}
\end{pmatrix}\]
where $I_{i_j}$ is the identity matrix of size $i_j$,  $a_{i_j}\in\mathbb C^{\ast}$, $a_{i_j}\not=a_{i_k}$ for $1\leq i_j,i_k\leq m$, and $\sum_{j=1}^mi_j=n+2$.
We assume that Fix$(g)$ contains a hypersurface $H$ in $X$.
Since dim\,$H=n-1$, $i_j\geq n-1$ for some $1\leq j\leq m$.
Then we may assume that 
\[g=
\begin{pmatrix}
a &0 \\
0&I_{n+1}
\end{pmatrix}\  {\rm or}\  
\begin{pmatrix}
a &0&0 \\
0&b&0\\
0&0&I_{n}
\end{pmatrix}.\]
If $g$ is defined by the former matrix, then by Theorem \ref{thm:8} $X$ has a Galois point $p$, and $g$ is an automorphism belonging to the Galois point $p$.

From here, we will show that if $g$ is defined by the latter matrix, then $n=2$, ${\rm ord}(g)=d-1$, and curves contained in Fix$(g)$ are $\mathbb p^1$.
By the representation matrix of $g$, $H=\{X_0=0\}\cap \{X_1=0\}$.
Let $F(X_0,\ldots ,X_{n+2})$ be the defining equation of $X$.
Since $H=\{X_0=0\}\cap \{X_1=0\}$, 
\begin{equation*}
\begin{split}
F(X_0,\ldots ,X_{n+2})=&F_{1,0}(X_2,\ldots,X_{n+2})X_0+F_{0,1}(X_2,\ldots,X_{n+2})X_1\\
&\ \ +\sum_{2\leq i+j\leq d}F_{i,j}(X_2,\ldots,X_{n+2})X^i_0X^j_1.
\end{split}
\end{equation*}
Since $X$ is smooth, $\{F_{1,0}(X_2,\ldots,X_{n+2})=0\}\cap\{F_{0,1}(X_2,\ldots,X_{n+2})=0\}\cap\{X_0=0\}\cap\{X_1=0\}=\emptyset$.
Therefore, $n=2$, curves of Fix$(g)$ are $\mathbb p^1$, $F_{1,0}(X_2,\ldots,X_{n+2})\not=0$, and $F_{0,1}(X_2,\ldots,X_{n+2})\not=0$.
Then $a=b$.
If ${\rm ord}(g)=d$, then $a=b=e_d$.
Then the defining equation of $X$ is as follows. 
\[ F(X_0,\ldots ,X_{n+2})=F_{1,0}(X_2,\ldots,X_{n+2})X_0+F_{0,1}(X_2,\ldots,X_{n+2})X_1.\]
Points $[1:0:0:0]$ and $[0:1:0:0]$ are singular points of $X$.
This contradicts that $X$ is smooth.
Therefore, ${\rm ord}(g)=d-1$.
\end{proof}
In the same way, we get Theorem \ref{thm:5} (Theorem \ref{thm:11}). 
\begin{thm}\label{thm:11}
Let $X$ be a smooth hypersurface of degree $d$ in $\mathbb P^{n+1}$, $g\in{\rm Aut}(X)$ be a linear automorphism of order $k(d-1)$ for $k\geq 2$.\\
(1) If $n=2$ and $\sharp|{\rm Fix}(g)|\geq5$, then $X$ has an inner Galois point $p$, and $g^k$ is an automorphism belonging to the Galois point $p$.\\
% such that $G$ contains the Galois group of $\pi_p:X\dashrightarrow \mathbb P^2$.\\
%(2) If $n=3$ and ${\rm Fix}(g)$ contains a curve $C'$ which is not a smooth rational curve $\mathbb P^1$, then $X$ has an inner $(resp. outer)$ Galois point $p$ such that $G$ contains the Galois group of $\pi_p:X\dashrightarrow \mathbb P^3$.\\
(2) If $n\geq 3$, and the dimension of ${\rm Fix}(g)$ is $n-2$, then $X$ has an inner Galois point $p$, and $g^k$ is an automorphism belonging to the Galois point $p$.
%contains the Galois group of $\pi_p:X\dashrightarrow \mathbb P^n$.
\end{thm}
\begin{proof}
As like the proof of Theorem \ref{thm:10}, we may assume that
\[g=
\begin{pmatrix}
a&0&0 \\
0&b&0\\
0&0&I_{n}
\end{pmatrix}\   {\rm or}\   
\begin{pmatrix}
a&0&0&0\\
0&b&0&0\\
0&0&c&0\\
0&0&0&I_{n-1}
\end{pmatrix}\]
where $a$, $b$, $c$, and $1$ are different numbers from each other.

First, we will show that if $g$ is defined by the former matrix, then $X$ has an inner Galois point $p$, and 
$g^k$ is an automorphism belonging to the Galois point $p$.

Let $F(X_0,\ldots ,X_{n+2})$ be the defining equation of $X$.
Since dim${\rm Fix}(g)=n-2$, 
\[
F(X_0,\ldots ,X_{n+2})=\sum_{1\leq i+j\leq d}F_{i,j}(X_2,\ldots,X_{n+2})X^i_0X^j_1+G(X_2,\ldots,X_{n+2})\]
where $G(X_2,\ldots,X_{n+2})\not=0$.
Let $n(g):=\sharp|\{[1:0:\cdots:0],[0:1:0:\cdots:0]\}\cap X|$.

If $n(g)=0$, then $\sum_{1\leq i+j\leq d}F_{i,j}(X_2,\ldots,X_{n+2})X^i_0X^j_1$ has $X_0^d$ and $X_1^d$ terms.
Since $G(X_2,\ldots,X_{n+2})\not=0$, $a^d=b^d=1$.
This contradicts that ord$(g)>d$.

If $n(g)=1$, then we may assume that $\sum_{1\leq i+j\leq d}F_{i,j}(X_2,\ldots,X_{n+2})X^i_0X^j_1$ has (i) $X_0^d$ and $X_iX_1^{d-1}$ terms, or (ii) $X_0^d$ and $X_0X_1^{d-1}$ terms where $2\leq i\leq n+2$.
The case (i) implies that $a^d=b^{d-1}=1$.
By Theorem \ref{thm:8}, there is an inner Galois point $p$ of $X$, and $g^k$ is an automorphism belonging to the Galois point $p$.
The case (ii) implies that $a^d=ab^{d-1}=1$.
Same as above, $X$ has an inner Galois point $p$, and $g^k$ is an automorphism belonging to the Galois point $p$.

If $n(g)=2$, then we may assume that $\sum_{1\leq i+j\leq d}F_{i,j}(X_2,\ldots,X_{n+2})X^i_0X^j_1$ has (iii) $X_iX_0^d$ and $X_iX_1^{d-1}$ terms, (iv) $X_iX_0^d$ and $X_0X_1^{d-1}$,
or (v) $X_1X_0^{d-1}$ and $X_0X_1^{d-1}$ terms where $2\leq i,j\leq n+2$.
The case (iii) implies that $a^{d-1}=b^{d-1}=1$.
This contradicts that ord$(g)>d-1$.
As like the case $n(g)=1$, if the case is (iv), then
by Theorem \ref{thm:8}, there is an inner Galois point $p$ of $X$, and $g^k$ is an automorphism belonging to the Galois point $p$.
The case (v) implies that $a^{d-1}b=ab^{d-1}=1$.
Then ord$(g)$ divides $(d-2)d$.
This contradicts that ord$(g)=k(d-1)$.

From here, we study the latter case, i.e. $g=
\begin{pmatrix}
a&0&0&0\\
0&b&0&0\\
0&0&c&0\\
0&0&0&I_{n-1}
\end{pmatrix}$.
As like the proof of Theorem \ref{thm:10}, we get that $n\leq3$.
We assume that $n=3$.
Let $F(X_0,\ldots ,X_5)$ be the defining equation of $X$.
Since the dimension of ${\rm Fix}(g)$ is $n-2$,
\begin{equation*}
\begin{split}
F(X_0,\ldots ,X_{n+2})=&\sum_{i=0}^2F_i(X_3,\ldots,X_{n+2})X_i\\
&\ \ +\sum_{2\leq i+j+k\leq d}F_{i,j,k}(X_3,\ldots,X_{n+2})X^i_0X^j_1X_2^k.
\end{split}
\end{equation*}
Since $X$ is smooth, $F_i(X_3,\ldots,X_{n+2})\not=0$ for $i=0,1,2$.
Then $a=b=c$.
This contradicts that ord$(g)=k(d-1)$ for $k\geq2$.
Then $n=2$, and hence
$g=
\begin{pmatrix}
a&0&0&0\\
0&b&0&0\\
0&0&c&0\\
0&0&0&1
\end{pmatrix}$.
Since $a$, $b$, $c$, and $1$ are different numbers from each other,
$\sharp|{\rm Fix}(g)|\leq 4$.
From the above, we get this theorem.
\end{proof}
The following example shows that Theorem \ref{thm:11} does not hold for an outer Galois point.
\begin{exa}\label{exa:6}
Let $d_1\geq 7$ be an odd integer, and $d:=2d_1+1$.
Let $X$ be a smooth hypersurface of degree $d$ in $\mathbb P^4$ defined by 
\[X_0^d+X_0^{\frac{d+1}{2}}X_1^{\frac{d-1}{2}}+X_0X_1^{d-1}+
X_2^{d-1}X_4+X_2X_3^{d-1}+X_3X_4^{d-1}=0.\]
The $X$ has an automorphism $g$ of order $\frac{(d-1)}{2}d$ such that 
the following matrix $A$ is a representation matrix of $g$:
\[
A:=
\begin{pmatrix}
e_{\frac{(d-1)}{2}d}^{1-d}&0&0&0&0\\
0&e_{\frac{(d-1)}{2}d}&0&0&0\\
0&0&1&0&0\\
0&0&0&1&0\\
0&0&0&0&1
\end{pmatrix}.
\]
Then the dimension of Fix$(g^{d^2-3d+3})$ is $1$.
In addition,
$X$ has an automorphism $h$ such that the following matrix $B$ is a representation matrix of $h$:
\[
B:=
\begin{pmatrix}
	e_{\frac{(d-1)}{2}d}^{1-d}&0&0&0&0\\
	0&e_{\frac{(d-1)}{2}d}&0&0&0\\
	0&0&e_{d^2-3d+3}&0&0\\
	0&0&0&e_{d^2-3d+3}^{d-1}&0\\
	0&0&0&0&1
\end{pmatrix}.
\]
If $3$ divides $d$, then ord$(h)=\frac{(d-1)}{6}d(d^2-3d+3)$, and if $3$ does not divide $d$, then ord$(h)=\frac{(d-1)}{2}d(d^2-3d+3)$. 
For $1\leq i<\frac{d-1}{2}$, the diagonal entries of $B^i$ are different from each other.
By Theorem \ref{thm:6}, $\delta(X)\leq 2$ and $\delta'(X)\leq5$.
Since $\frac{d-1}{2}\geq 7$, if $X$ has a Galois point, then there is a Galois point $p$ of $X$ such that $g^l(p)=p$ for some $1\leq l<\frac{(d-1)}{2}d$.
As like Example \ref{exa:1}, this is a contradiction.
Then $X$ does not have Galois points.
\end{exa}

\end{document}